\documentclass[pdftex, a4paper, 12pt, oneside]{amsart}

\usepackage{geometry}
\usepackage[cp1250]{inputenc}

\usepackage{alltt}
\usepackage{euler}
\usepackage{amsfonts}
\usepackage{amscd}
\usepackage{graphicx}
\usepackage{lpic}
\usepackage{longtable}
\usepackage{url}
\usepackage{import}
\usepackage{amsmath, amscd, amsfonts, amsthm, amssymb, mathrsfs}
\usepackage[sc]{mathpazo}
\usepackage{enumerate}

\usepackage{tikz}
\usetikzlibrary{intersections}
\usepackage{pgfplots}
\usepgfplotslibrary{external}

\allowdisplaybreaks

\numberwithin{equation}{section}

\theoremstyle{plain}
\newtheorem{theorem}{Theorem}[section]
\newtheorem{proposition}[theorem]{Proposition}

\newtheorem{corollary}[theorem]{Corollary}

\theoremstyle{definition}
\newtheorem{definition}[theorem]{Definition}
\newtheorem{example}[theorem]{Example}
\newtheorem{remark}[theorem]{Remark}

\usepackage{booktabs}

\usepackage[pagebackref=false]{hyperref}

\usepackage[all]{hypcap}

\definecolor{newblue}{rgb}{0.27, 0.32, 0.86}
\definecolor{newred}{rgb}{0.86, 0.32, 0.27}
\hypersetup{
	pagebackref=true,
	colorlinks=true,       % false: boxed links; true: colored links
	linkcolor=newred,          % color of internal links (change box color with linkbordercolor)
	citecolor=newblue,        % color of links to bibliography
	filecolor=magenta,      % color of file links
	urlcolor=newblue           % color of external links
}

\usepackage{caption}
\usepackage{wrapfig}
\providecommand{\keywords}[1]{\textbf{\textit{Key words and phrases:}} #1}
\providecommand{\subjclass}[1]{\textbf{\textit{2020 Mathematics Subject Classification:}} #1}

\title[\tiny Minimal generating sets of moves for surfaces immersed surfaces]{Minimal generating sets of moves for surfaces\\ immersed in the four-space}
\author{Micha\l \;Jab\l onowski}
\address{Institute of Mathematics, Faculty of Mathematics, Physics and Informatics, University of Gda\'nsk, 80-308 Gda\'nsk, Poland}
\email{michal.jablonowski@gmail.com}

\subjclass{57K45, Secondary: 57Q35, 57R42} 

\keywords{immersed surface, marked graph diagram, surface-link, link with bands, minimal set of moves}

\date{\today}

\graphicspath{ {fig/} }

\begin{document}
	
\maketitle

\begin{abstract}
For immersed surfaces in the four-space, we have a generating set of the Swenton--Hughes--Kim--Miller spatial moves that relate singular banded diagrams of ambient isotopic immersions of those surfaces. We also have Yoshikawa--Kamada--Kawauchi--Kim--Lee planar moves that relate marked graph diagrams of ambient isotopic immersions of those surfaces. One can ask if the former moves form a minimal set and if the latter moves form a generating set. In this paper, we derive a minimal generating set of spatial moves for diagrams of surfaces immersed in the four-space, which translates into a generating set of planar moves. We also show that the complements of two equivalent immersed surfaces can be transformed one another by a Kirby calculus not requiring the 1-1-handle or 2-1-handle slides. We also discuss the fundamental group of the immersed surface-link complement in the four-space and a quandle coloring invariant of an oriented immersed surface-link.
\end{abstract}

%\tableofcontents

\section{Introduction}

There is a set of four spatial moves: \emph{cup, cap, band slide, band swim} for surface-links embedded in $\mathbb{R}^4$ proven by F.J.~Swenton \cite{Swe01} (and later clarified by C.~Kearton and V.~Kurlin \cite{KeaKur08}), to be a generating set of spatial moves between any diagrams of equivalent embedded surface-links. It was proved later to be a minimal generating set \cite{Jab20}. It was recently extended by a set of three moves \emph{intersection/band slide, intersection/band pass, intersection/band swim} by M.~Hughes, S.~Kim, and M.~Miller \cite{HKM21} to form a generating set of moves relating equivalent immersed surface-links.
\par
There is a set of $11$ planar moves: $\{\text{planar isotopy, }\Omega_1, \ldots, \Omega_8, \Omega_4',\Omega_6'\}$ for surface-links embedded in $\mathbb{R}^4$  introduced by K.~Yoshikawa \cite{Yos94}, and recently extended by a set of three moves $\{\Omega_9, \Omega_9',\Omega_{10}\}$ by S.~Kamada, A.~Kawauchi, J.~Kim, and S.Y.~Lee \cite{KKKL18} to include moves relating equivalent immersed surface-links.
\par
It is natural to ask if the set $\mathcal A = \{$spatial isotopy, cup, cap, band slide, band swim, intersection/band slide, intersection/band pass, intersection band swim$\}$ is a minimal set, in the sense that none of these moves are a combination of other moves from the set $\mathcal A$. We negatively answer this question and show that exactly one of those moves is not necessary.

\begin{theorem}\label{tw:mainA}
	The set $\mathcal C = \{\text{spatial isotopy, cup, cap, band slide, band swim},\\ \text{intersection/band slide, intersection/band pass}\}$ is a minimal generating set for $\mathcal A$.
\end{theorem}

It is also natural to ask if the set $\mathcal B=\{\text{planar isotopy, }\Omega_1, \ldots, \Omega_{10}, \Omega_4',\Omega_6',\Omega_9'\}$ of $14$ moves is a generating set, in the sense that it relates any two marked graph diagrams of equivalent immersed surface-links. We propose a different set of moves $$\mathcal D=\{\text{planar isotopy, }\Omega_1, \ldots, \Omega_{12}, \Omega_4',\Omega_6', \Omega_9',\Omega_{11}',\Omega_{12}'\},$$ where the moves $\{\Omega_1, \ldots, \Omega_8, \Omega_4',\Omega_6'\}$ agree with the Yoshikawa moves, and new $\{\Omega_9, \Omega_9',\Omega_{10},\Omega_{11},\Omega_{11}',\Omega_{12},\Omega_{12}'\}$ moves and prove their sufficiency as follows.

\begin{theorem}\label{tw:mainB}
	Two surface-links immersed in $\mathbb{R}^4$ are equivalent if and only if, their singular marked graph diagrams are related by a combination of moves from $\mathcal D$.
\end{theorem}

This paper is organized as follows. In Section \ref{sec1}, we give preliminaries of surface-links and the motion picture method. In Section \ref{sec2} we review the singular banded unlink method of presenting immersed surface-links and we prove Theorem \ref{tw:mainA}. In Section \ref{sec3} we review the singular marked graph diagram method of presenting immersed surface-links and we prove Theorem \ref{tw:mainB}. In Section \ref{sec4} we review the group of an immersed surface-link with a method to obtain its presentation from a singular marked graph diagram. In Section \ref{sec5} we derive oriented planar moves for the oriented-surface case. In Section \ref{sec6} we discuss a a quandle coloring invariant for an oriented immersed surface-link.

\section*{Acknowledgements}

The author would like to thank the referee for carefully
reading the paper. "This research was funded in whole or in part by NCN 2023/07/X/ST1/00157. For the purpose of Open Access, the author has applied a CC-BY public copyright licence to any Author Accepted Manuscript (AAM) version arising from this submission."

\section{Preliminaries}\label{sec1}

Let $X,Y$ be smooth ($C^{\infty}$) manifolds. Let $f:X^n\to Y^{m}$ be a smooth map. It is called an \emph{immersion} if at each point $x\in X$ the induced differential $df$ (a map between tangent spaces) is a monomorphism. By the Whitney immersion theorem, any smooth map $f:X^n\to Y^{m}$ can be approximated homotopically with arbitrary accuracy by an immersion when $m\geq 2n$. We consider smooth immersions $f:X^n\to Y^{2n}$ such that the following three conditions are satisfied: (i) $\#|f^{-1}(f(x))|\leq 2$, (ii) there is only finite number of points with $\#|f^{-1}(f(x))|= 2$, (iii) at each singularity $p=f(x)=f(y)$, there is a coordinate chart around $p$ where the two coordinate subspaces $\mathbb{R}^n\times 0$ and $0 \times\mathbb{R}^n$ are exactly the immersed images of $f$ near $x$ and $y$ respectively. That is the map is "self-transverse". By general position theorems for maps, any smooth map $f:X^n\to Y^{2n}$ can be approximated homotopically with arbitrary accuracy by an immersion described above.
\par
Such an immersion (or its image when no confusion arises) of a closed (i.e. compact, without boundary) surface $F$ into the Euclidean $\mathbb{R}^4$ (or into the $\mathbb{S}^4=\mathbb{R}^4\cup\{\infty\}$) is called an \emph{immersed surface-link} (or \emph{immersed surface-knot} if it is connected). 
\par
Two immersed surface-links are \emph{equivalent} if there exists an orientation preserving homeomorphism of the four-space $\mathbb{R}^4$ to itself (or equivalently auto-homeomorphism of the four-sphere $\mathbb{S}^4$), mapping one of those surfaces onto the other. We will use a word \emph{classical} referring to the theory of embeddings of circles $S^1\sqcup\ldots\sqcup S^1\hookrightarrow \mathbb{R}^3$ modulo ambient isotopy in $\mathbb{R}^3$ with their planar or spherical regular projections.
\par
To describe immersed surface-links in $\mathbb{R}^4$, we will use \emph{hyperplane cross-sections} $\mathbb{R}^3\times\{t\}\subset\mathbb{R}^4$ for $t\in\mathbb{R}$, denoted by $\mathbb{R}^3_t$. This method (called \emph{motion picture method}) introduced by Fox and Milnor was presented in \cite{Fox62}. By a general position argument, the intersection of $\mathbb{R}^3_t$ and a surface-link $F$ can (except in finite cases) be either empty or a classical link. In the finite cases, the intersection can be a single point (corresponding to a \emph{maximal point} or to a \emph{minimal point}) or a four-valent embedded graph, where each vertex corresponds to a \emph{saddle point} or to a \emph{singular point}. For more introductory material on this topic in an embedded case refer to \cite{Kam17}, and in the immersed case refer to \cite{HKM21}.

\section{Singular banded unlinks}\label{sec2}

\subsection{Links with bands}

\begin{definition}
	A \emph{singular link} $L$ in $\mathbb{R}^3$ is the image of an immersion in the classical case $\iota: S^1 \sqcup \cdots \sqcup S^1 \rightarrow \mathbb{R}^3$ which is injective except at isolated double points that are not tangencies (i.e. a $4$-valent fat-vertex graph smoothly embedded in $\mathbb{R}^3$.). At every double point $p$ we include a small disk $v\cong D^2$ embedded in $\mathbb{R}^3$ such that $(v,v \cap L) \cong (D^2,\{(x,y)\in D^2\,|\,xy=0\})$.  We refer to these disks as the {\emph{vertices}} of $L$. The double points of a singular link $L$ correspond to the isolated double points of an immersed surface in $\mathbb{R}^4$.
	\par
	A \emph{marked singular link} $(L, \sigma)$ in $\mathbb{R}^3$ is a singular link $L$ along with decorations $\sigma$ on the vertices of $L$, as follows: say that $v$ is a vertex of $L$, with $\partial v\cap \overline{(L\setminus v)}$ consisting of the four points $p_1, p_2, p_3, p_4$ in cyclic order. Choose a co-orientation of the disk $v$. On the positive side of $v$, add an arc connecting $p_1$ and $p_3$. On the negative side of $v$, add an arc connecting $p_2$ and $p_4$. A choice of $\sigma$ involves making a fixed choice of decoration on $v$, for all vertices $v$ of $L$.
\end{definition}

Let $L_+$ denote the classical link obtained from $(L,\sigma)$ by pushing the arc of $L$ between $p_1$ and $p_3$ off $v$ in the positive direction, and repeating for each vertex in $L$. We call $L_+$ the {\emph{positive resolution}} of $(L,\sigma)$. Similarly, let $L_-$ denote the classical link obtained from $(L,\sigma)$ by pushing the arc of $L$ between $p_1$ and $p_3$ off $v$ in the negative direction, and repeating for each vertex in $L$. We call $L_-$ the {\emph{negative resolution}} of $(L,\sigma)$ (see Figure \ref{michal_11}).

\begin{figure}
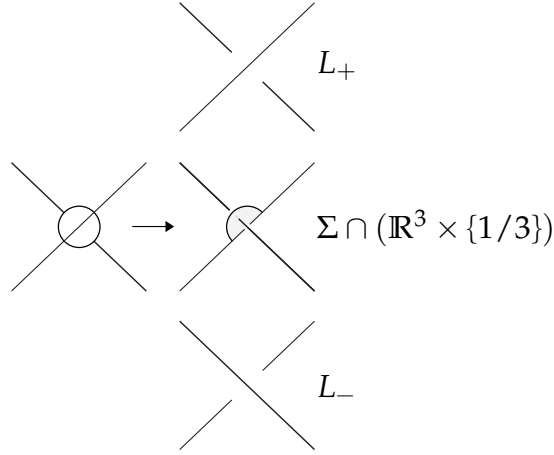

	\begin{center}
		\begin{lpic}[r(3cm)]{michal_11(4cm)}
			\lbl[l]{75,95;$L_+$}
			\lbl[l]{75,55;$\Sigma \cap (\mathbb{R}^3 \times \{1/3\})$}
			\lbl[l]{75,15;$L_-$}
		\end{lpic}
		
		\caption{A vertex of a marked singular link and the corresponding surface cuts (\protect\cite{HKM21}).\label{michal_11}}
	\end{center}
\end{figure}

Let $L$ be a singular link, and let $\delta_L$ denote the union of the vertices of $L$.  A \emph{band} $b$ attached to $L$ is the image of an embedding $\phi : I \times I \hookrightarrow \mathbb{R}^3\backslash \delta_L$, where $I = [-1,1]$, and $b \cap L = \phi (\{ -1,1\} \times I)$. Let $L_b$  be the singular link defined by 
\[
L_b = (L \backslash \phi (\{-1,1\} \times I)) \cup \phi (I \times \{-1,1\}).
\]
Then we say that $L_b$ is the result of performing \emph{band surgery} to $L$ along $b$.  If $B$ is a finite family of pairwise disjoint bands for $L$, then we will let $L_B$ denote the link we obtain by performing band surgery along each of the bands in $B$.  We say that $L_B$ is the result of \emph{resolving} the bands in $B$. A \emph{singular link with bands} $LB$ in $\mathbb R^3$ is a pair $(L, \sigma, B)$ consisting of a singular link $L$ in $\mathbb R^3$ with a decoration $\sigma$ and a finite set $B=\{b_1, \dots, b_n\}$ of pairwise disjoint $n$ bands spanning $L$.
\par
Let $(L, \sigma, B)$ be a singular link with bands $B=\{b_1, \dots, b_n\}$ and $\Delta_1,\ldots,\Delta_a \subset \mathbb{R}^3$ be mutually disjoint $2$-disks with $\partial(\cup_{j=1}^a\Delta_j)= L_{B+}$, and let $\Delta_1',\ldots,\Delta_b' \subset \mathbb{R}^3$ be mutually disjoint $2$-disks with $\partial(\cup_{k=1}^b\Delta_k')= L_-$. For each vertex $v$ of $L$, these two opposite push-offs form a bigon in a neighborhood of $v$, which bounds an embedded disk $D_v$.  This disk $D_v$ can be chosen so that its interior intersects $L$ transversely in a single point near $v$.  For each vertex $v$ select such a disk $D_v$ (ensuring that all of these disks are pairwise disjoint) and let $D_L$ denote the union of all of these embedded disks.
\par
We define $\Sigma \subset \mathbb{R}^3 \times \mathbb{R} = \mathbb{R}^4 $ an \emph{immersed surface-link corresponding to} $(L, \sigma, B)$ by the following cross-sections.

$$
(\mathbb{R}^3_t, \Sigma \cap \mathbb{R}^3_t)=\left\{%
\begin{array}{ll}
	(\mathbb R^3, \emptyset) & \hbox{for $t > 1$,}\\
	(\mathbb R^3, L_{B+} \cup (\cup_{j=1}^a\Delta_j)) & \hbox{for $t = 1$,} \\
	(\mathbb R^3, L_{B+}) & \hbox{for $0 < t < 1$,} \\
	(\mathbb R^3, L_+\cup (\cup_{i=1}^n b_i)) & \hbox{for $t = 0$,} \\
	(\mathbb R^3, L_+) & \hbox{for $-1/2 < t < 0$,} \\
	(\mathbb R^3, L_- \cup D_L) & \hbox{for $t = -1/2$,} \\
	(\mathbb R^3, L_-) & \hbox{for $-1 < t < -1/2$,} \\
	(\mathbb R^3, L_- \cup (\cup_{k=1}^b\Delta_k')) & \hbox{for $t = -1$,} \\
	(\mathbb R^3, \emptyset) & \hbox{for $ t < -1$.} \\
\end{array}
\right.
$$

It is known that the singular surface-link type of $\Sigma$ does not depend on choices of trivial disks (cf. \cite{KSS82}). The \emph{singular band moves} (illustrated in Figures \ref{michal_03}--\ref{michal_09}) on a singular link with bands are:

\begin{enumerate}
	\item[$(M_0)$] spatial isotopy (i.e. isotopy in $\mathbb{R}^3$),
	\item[$(M_1)$] cup moves,
	\item[$(M_2)$] cap moves,
	\item[$(M_3)$] band slides,
	\item[$(M_4)$] band swims,
	\item[$(M_5)$] passing a vertex past the edge of a band,
	\item[$(M_6)$] sliding a vertex over a band,
	\item[$(x)$] swimming a band through a vertex.
\end{enumerate}

\begin{figure}[ht]
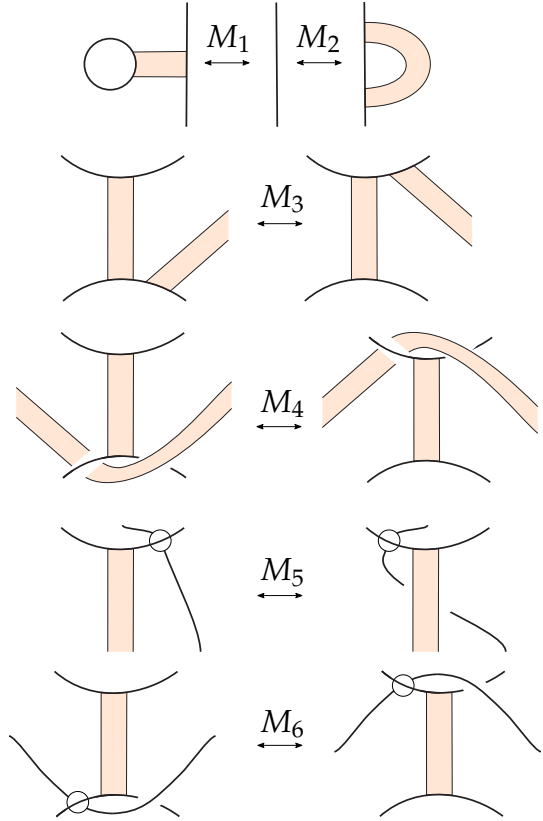

	\begin{center}
		\begin{lpic}[]{michal_03(7cm)}
			\lbl[b]{49,172;$M_1$}
			\lbl[b]{69,172;$M_2$}
			\lbl[b]{61,136;$M_3$}
			\lbl[b]{61,90;$M_4$}
			\lbl[b]{61,52;$M_5$}
			\lbl[b]{61,18;$M_6$}
		\end{lpic}
		\caption{Singular band moves $\{M_1, \ldots, M_6\}$.\label{michal_03}}
	\end{center}
\end{figure}

Let $\mathcal A = \{\text{spatial isotopy, cup, cap, band slide, band swim},\\ \text{intersection/band slide, intersection/band pass, intersection/band swim}\}$.

\begin{theorem}\cite{HKM21}
	Two surface-links immersed in $\mathbb{R}^4$ are equivalent if and only if, their banded unlink diagrams are related by a combination of moves from the set $\mathcal A$.
\end{theorem}

It is natural then to ask if the set $\mathcal A$ is a minimal set, in the sense that none of these moves are a combination of other moves from the set $\mathcal A$. We answer this question by showing that the set $\mathcal C = \{M_0, \ldots, M_6\}$ is a minimal generating set for $\mathcal A$.

\begin{proof}[Proof of theorem \ref{tw:mainA}]
	We will show that the move $(x)$ intersection/band swim, as in Figure \ref{michal_09} can be generated by moves from $\mathcal C$, and that each move from $\mathcal C$ is independent from other moves from $\mathcal C$.
	
	\begin{figure}[ht]
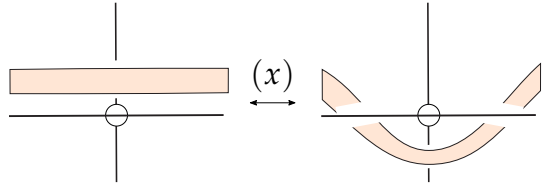

		\begin{center}
			\begin{lpic}[]{michal_09(7cm)}
				\lbl[b]{58,20;$(x)$}
			\end{lpic}
			\caption{The singular band move $(x)$.\label{michal_09}}
		\end{center}
	\end{figure}
	
	First, we introduce a useful move $M_b$ shown in Figure \ref{michal_02} together with the proof that it is a combination of moves from $\mathcal C$. We can now obtain the move ($x$) intersection/band swim as a combination of moves from the set $\{\text{spatial isotopy}, M_4, M_b\}$ as in Figure \ref{michal_12}.
	
	\begin{figure}[ht]
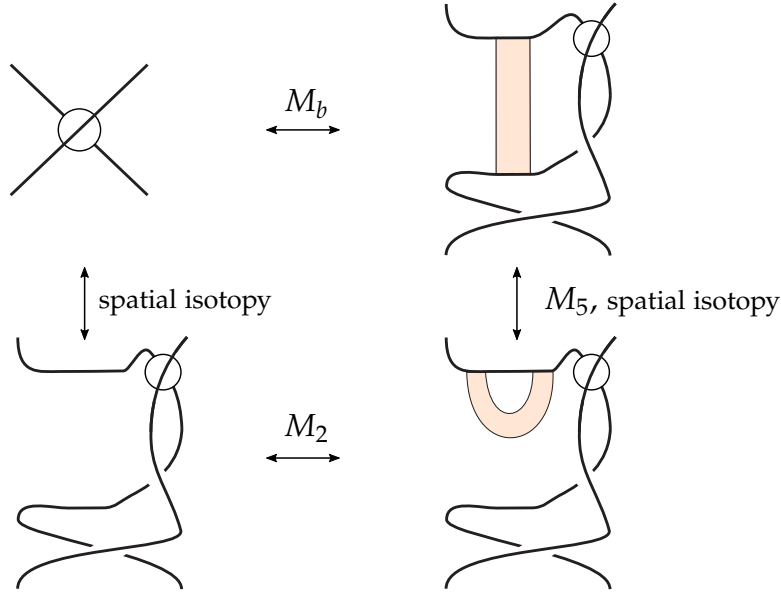

		\begin{center}
			\begin{lpic}[r(1.93cm)]{michal_02(8cm)}
				\lbl[b]{40,64;$M_b$}
				\lbl[b]{40,20;$M_2$}
				\lbl[l]{12,39; {\footnotesize spatial isotopy}}
				\lbl[l]{72,39; $M_5$, {\footnotesize spatial isotopy}}
				
			\end{lpic}
			\caption{The singular band move $M_b$.\label{michal_02}}
		\end{center}
	\end{figure}
	
	\begin{figure}[ht]
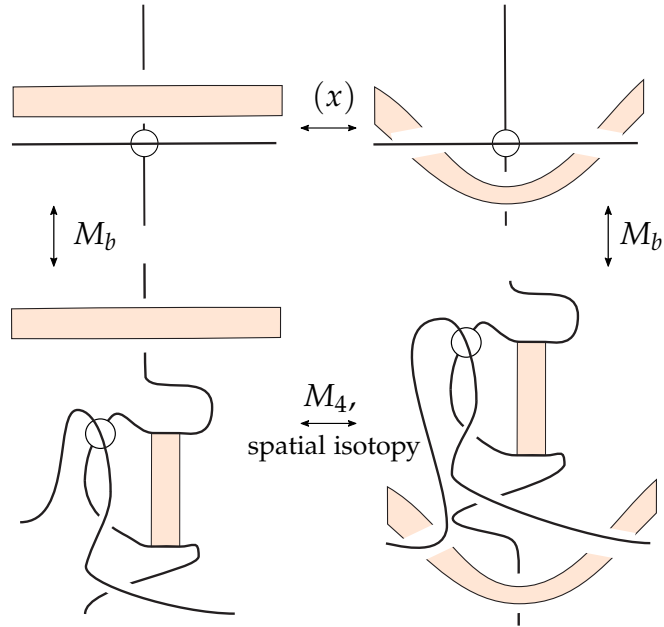

		\begin{center}
			\begin{lpic}[]{michal_12(8.5cm)}
				\lbl[b]{42,67;$(x)$}
				\lbl[b]{42,28;$M_4$, }
				\lbl[t]{42,25;{\footnotesize spatial isotopy}}
				\lbl[l]{8,51;$M_b$}
				\lbl[l]{79,51;$M_b$}
				
			\end{lpic}
			\caption{Deriving the move $(x)$.\label{michal_12}}
		\end{center}
	\end{figure}
	
	To show that each move from $\mathcal C' = \{M_0, M_1, M_2, M_3, M_4\}$ is independent from other moves from $\mathcal C$ we use the fact that $\mathcal C'$ is a minimal generating set of $\mathcal C$ proved in \cite{Jab20} and neither move from $\mathcal C'$ can be a combination of moves involving the move $M_5$ or $M_6$ because in the moves $M_5$, $M_6$ we have the number of singular vertices equal to one, before and after the move (and neither moves form $\mathcal C'$ have any singular vertices).
	\par
	It remains to show that each of the moves $M_5$ and $M_6$ is not a combination of the other moves from $\mathcal C$. We define two semi-invariants $f^k$ for $k\in\{M_5, M_6\}$ such that they preserve their values after performing each move from the set $\mathcal{C}\backslash \{k\}$, and we construct two pairs of singular links with bands $D_1^k, D_2^k$ of equivalent immersed surface-links such that $f^k(D_1^k)\not=f^k(D_2^k)$.
	\par
	Define $f^{M_5}$ and $f^{M_6}$ as an ambient isotopy class of classical links obtained from singular links with bands by changing all bands and singular vertices as shown in Figure \ref{michal_08} respectively, modulo the split sum with trivial knots. It is now not difficult to directly see that after performing the transformations on moves from $\mathcal C$ the property that the semi-invariants preserve their values after performing each move from the rest of the moves from the set $\mathcal{C}$.
	
	\begin{figure}[h]
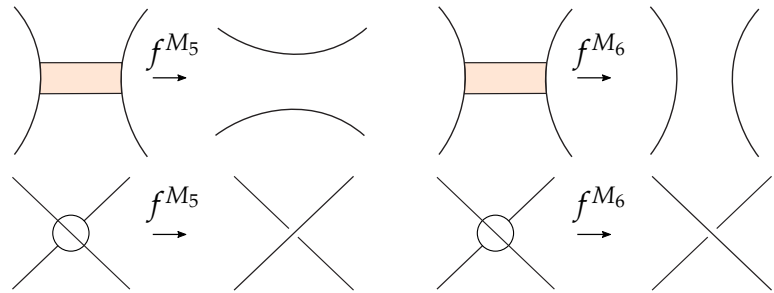

		\begin{center}
			\begin{lpic}[]{michal_08(10cm)}
				\lbl[b]{19,26;$f^{M_5}$}
				\lbl[b]{19,8;$f^{M_5}$}
				\lbl[b]{69,26;$f^{M_6}$}
				\lbl[b]{69,8;$f^{M_6}$}
			\end{lpic}
			\caption{Defining semi-invariants $f^{M_5}$ and $f^{M_6}$.\label{michal_08}}
		\end{center}
	\end{figure}
	
	It remains to show the diagrams $D_1^{M_5}, D_2^{M_5}$ and $D_1^{M_6}, D_2^{M_6}$ with the desired property explained above. They are illustrated in Figure \ref{michal_13}, where we have that $f^{M_5}(D_2^{M_5})$ and $f^{M_6}(D_1^{M_6})$ are ambient isotopy types of Hopf-links (i.e. nontrivial links), on the other hand, $f^{M_5}(D_1^{M_5})$ and $f^{M_6}(D_2^{M_6})$ are ambient isotopy types of trivial classical links.
	
	\begin{figure}[h]
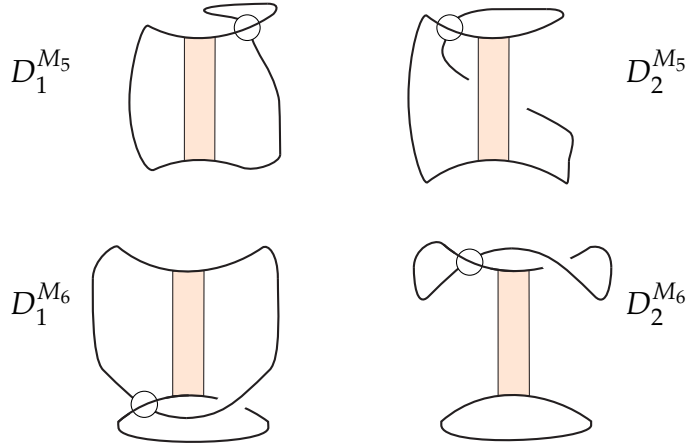

		\begin{center}
			\begin{lpic}[l(1.1cm),r(1.1cm)]{michal_13(6.85cm)}
				\lbl[r]{-2,40;$D_1^{M_5}$}
				\lbl[r]{-2,15;$D_1^{M_6}$}
				\lbl[l]{57,40;$D_2^{M_5}$}
				\lbl[l]{57,15;$D_2^{M_6}$}
			\end{lpic}
			\caption{The diagrams $D_1^{M_5}, D_2^{M_5}$ and $D_1^{M_6}, D_2^{M_6}$.\label{michal_13}}
		\end{center}
	\end{figure}
	
\end{proof}

\subsection{Immersed surface-link exterior}

Fix an immersed surface-link $F$ embedded in a manifold $\mathbb{S}^4$. For an open neighborhood, denoted $N(F)$, the \emph{exterior} of $F$ is $E(F):=\mathbb{S}^4\backslash N(F)$. If two singular surface-links are equivalent, then their exteriors are diffeomorphic. From a singular banded-unlink diagram of a surface-link $F$ we can create a Kirby diagram of the exterior of $F$ as follows (for the saddle point case see \cite[sec. 6.2]{GomSti99} and also \cite{Mar09} with swapped pre- and post-knot, for the singular point case see \cite[sec. 5.2]{Akb16} and also \cite{HKM21}).

\begin{proposition}\label{prop1}
	If $(L,\sigma, B)$ denotes a singular banded unlink decomposition of $F$ then a Kirby diagram of $E(F)$ can be created by decorating each component of $L$ with a dot, and replacing each band with a $0$-framed unknot linked with $L$ near the places of the band attachments to $L$, and replacing each singular point of $L$ as in Figure \ref{michal_10}.
	
\end{proposition}

\begin{figure}[h]
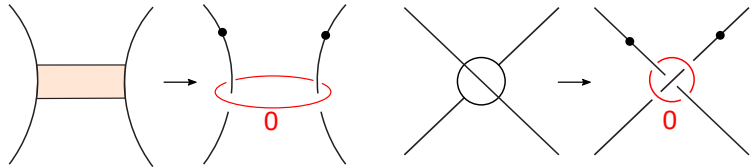

	\begin{center}
		\begin{lpic}[]{michal_10(9.75cm)}
			
		\end{lpic}
		\caption{A diagram for the exterior of a band and a singular point.\label{michal_10}}
	\end{center}
\end{figure}

\begin{remark}
	We also remark that the complements of two equivalent immersed surfaces can be transformed one another by a Kirby calculus not requiring $(1-1)$-handle or $(2-1)$-handle slides, see Figure \ref{michal_04}. This gives a potential room for a stronger immersed surface invariant (than the one requiring invariance of all the moves). If one finds the $4$-manifold invariant that preserves all handle attachment map isotopies, handle slides, and handle pair cancellations/creations except $(1-1)$-handle or $(2-1)$-handle slides then it will induce an invariant of immersed surface ambient isotopy.
\end{remark}

\begin{figure}[h!t]
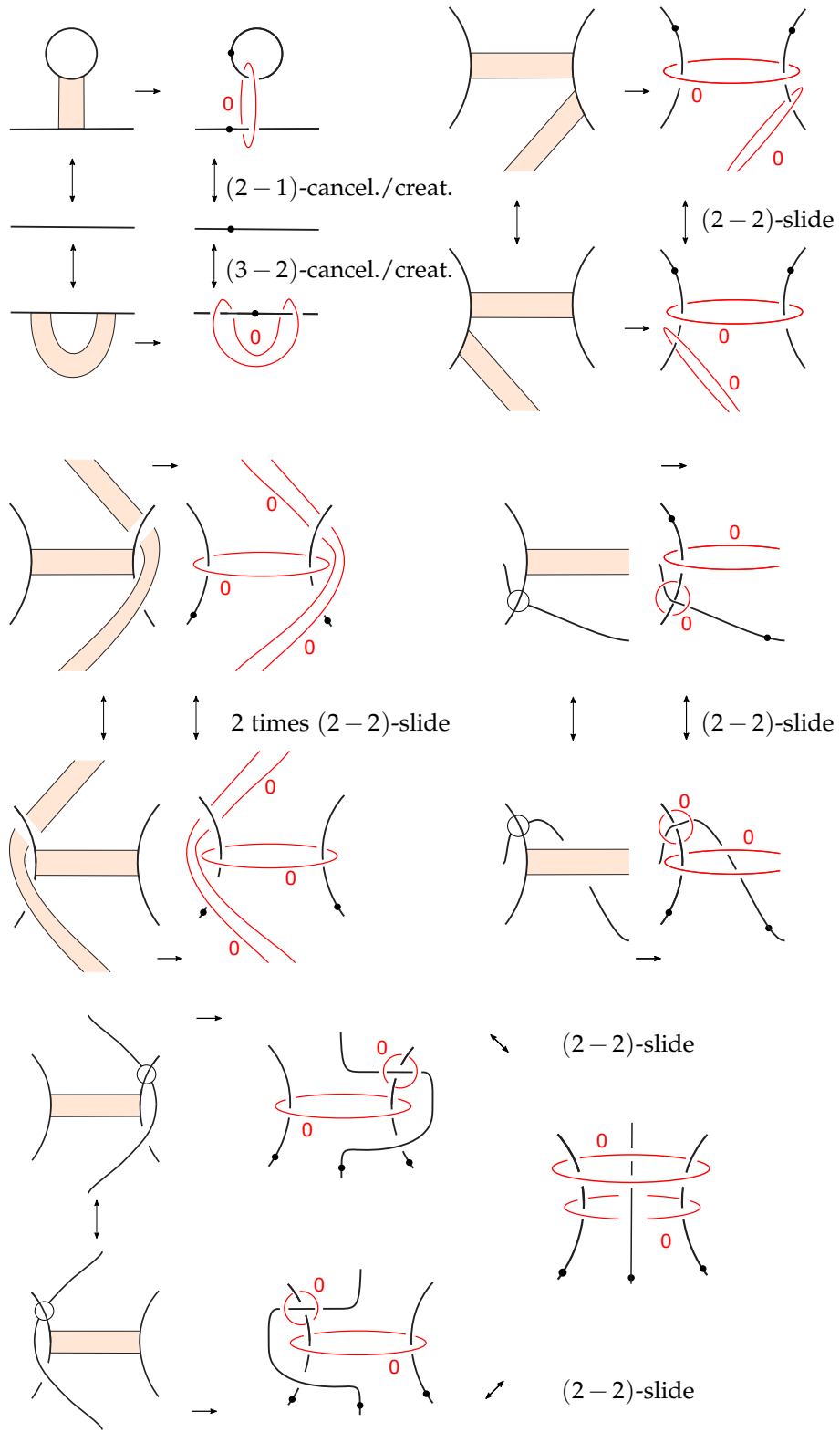

	\begin{center}
		\begin{lpic}[]{michal_04(11.5cm)}
			\lbl[l]{34,197;\footnotesize$(2-1)$-cancel./creat.}
			\lbl[l]{34,184;\footnotesize$(3-2)$-cancel./creat.}
			\lbl[l]{109,192;\footnotesize$(2-2)$-slide}
			\lbl[l]{35,112;\footnotesize$2$ times $(2-2)$-slide}
			\lbl[l]{109,112;\footnotesize$(2-2)$-slide}
			\lbl[l]{87,61;\footnotesize$(2-2)$-slide}
			\lbl[l]{87,6;\footnotesize$(2-2)$-slide}
		\end{lpic}
		\caption{Singular band moves and their corresponding Kirby moves.\label{michal_04}}
	\end{center}
\end{figure}

\section{Singular marked graph diagrams}\label{sec3}

A \emph{singular marked graph diagram} is a planar $4$-valent graph embedding, with the vertices decorated either by a classical crossing, marker or a singular decoration. For a singular marked graph diagram $D$, we denote by $L_{B+}(D)$ and $L_-(D)$ the classical link diagrams obtained from $D$ by transforming every vertex as presented in Fig.\;\ref{michal_06} for $+3/4$ and $-3/4$ case respectively.

\begin{figure}
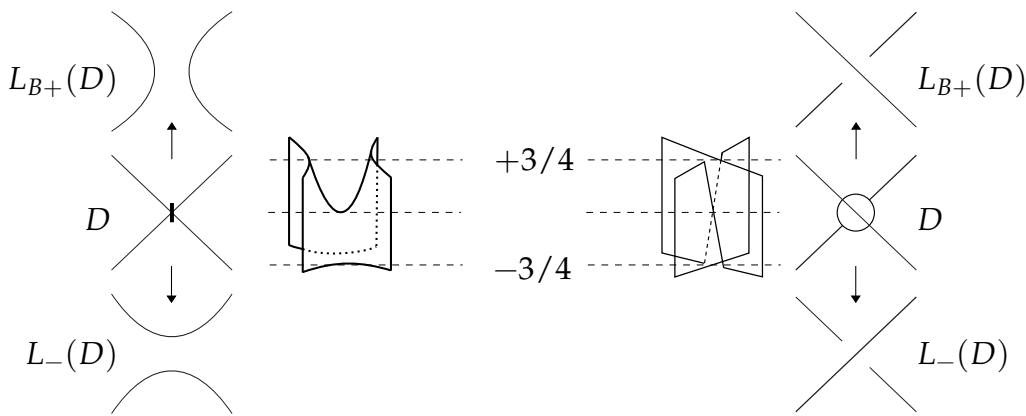

	\begin{center}
		\begin{lpic}[r(1.05cm),l(1.05cm)]{michal_06(10.6cm)}
			\lbl[l]{105,40;$-3/4$}
			\lbl[l]{105,70;$+3/4$}
			\lbl[r]{2,16;$L_-(D)$}
			\lbl[r]{0,54;$D$}
			\lbl[r]{2,92;$L_{B+}(D)$}
			\lbl[l]{221,16;$L_-(D)$}
			\lbl[l]{221,54;$D$}
			\lbl[l]{221,92;$L_{B+}(D)$}
		\end{lpic}
		\caption{A neighborhood in the four-space of a marker and a singular point.\label{michal_06}}
	\end{center}
\end{figure}

We call $L_{B+}(D)$ and $L_-(D)$ the \emph{positive resolution} and the \emph{negative resolution} of $D$, respectively.
\par
Any abstractly created singular marked graph diagram is an \emph{admissible diagram} if and only if its both resolutions are trivial classical link diagrams (see an example of an admissible diagram of an immersed sphere in Figure \ref{michal_14}).

\begin{figure}[h]
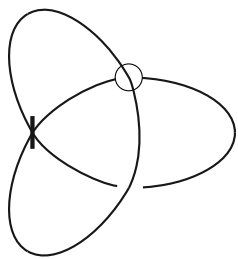

	
	\begin{center}
		\begin{lpic}[]{michal_14(3cm)}
			
		\end{lpic}
		\caption{An example.\label{michal_14}}
	\end{center}

\end{figure}

\begin{figure}[ht]
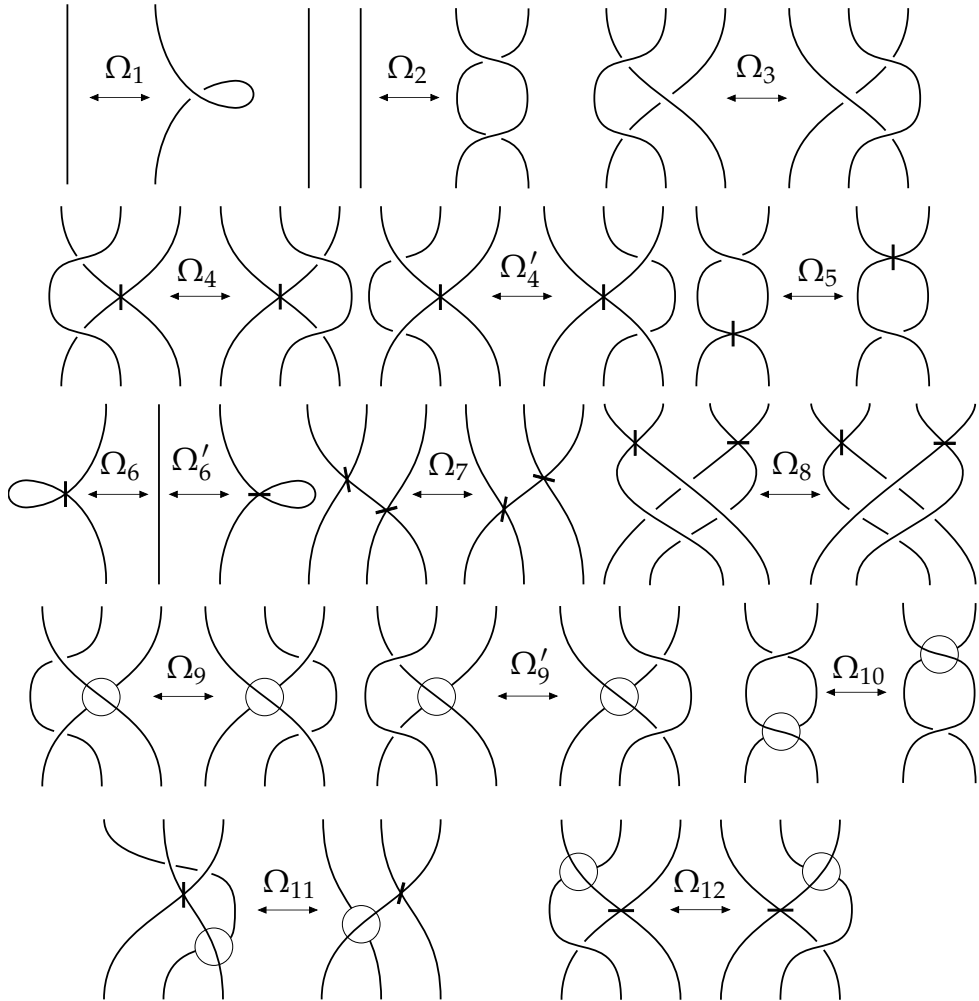

	\begin{center}
		\begin{lpic}[]{michal_05(12.75cm)}
			\lbl[b]{23,180;$\Omega_1$}
			\lbl[b]{78,180;$\Omega_2$}
			\lbl[b]{146,180;$\Omega_3$}
			\lbl[b]{37,140;$\Omega_4$}
			\lbl[b]{100,140;$\Omega_4'$}
			\lbl[b]{158,140;$\Omega_5$}
			\lbl[b]{22,102;$\Omega_6$}
			\lbl[b]{36,102;$\Omega_6'$}
			\lbl[b]{86,102;$\Omega_7$}
			\lbl[b]{153,102;$\Omega_8$}
			\lbl[b]{35,62;$\Omega_9$}
			\lbl[b]{102,62;$\Omega_9'$}
			\lbl[b]{166,62;$\Omega_{10}$}
			\lbl[b]{55,20;$\Omega_{11}$}
			\lbl[b]{135,20;$\Omega_{12}$}
		\end{lpic}
		\caption{A generating set of singular marked graph planar moves.\label{michal_05}}
	\end{center}
\end{figure}

\begin{figure}[ht]
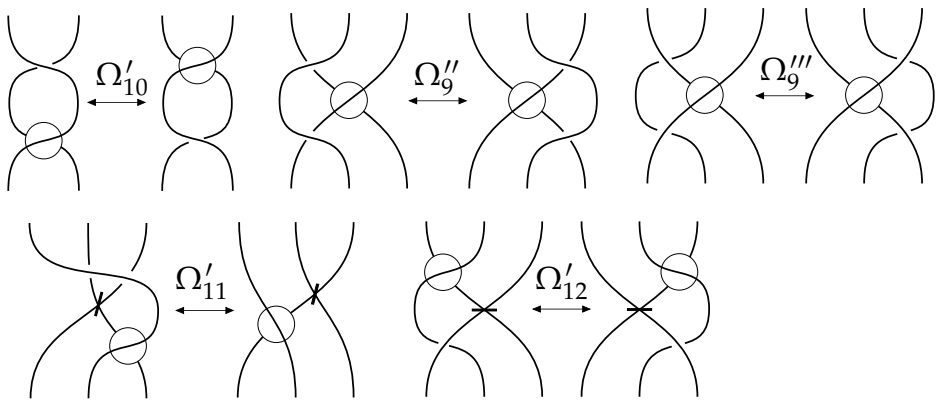

	\begin{center}
		\begin{lpic}[]{michal_05b(12.2cm)}
			\lbl[b]{23,61;$\Omega_{10}'$}
			\lbl[b]{86,61;$\Omega_{9}''$}
			\lbl[b]{156,62;$\Omega_{9}'''$}
			\lbl[b]{39,20;$\Omega_{11}'$}
			\lbl[b]{111,20;$\Omega_{12}'$}
		\end{lpic}
		\caption{Other singular marked graph planar moves.\label{michal_05b}}
	\end{center}
\end{figure}

\subsection{Singular marked graph diagram moves}

We propose the set of moves\\ $\mathcal D=\{\text{planar isotopy, }\Omega_1, \ldots, \Omega_{12}, \Omega_4',\Omega_6', \Omega_9',\Omega_{11}',\Omega_{12}'\}$ on admissible diagrams, shown in Figure \ref{michal_05}, where the set of moves $\{\Omega_1, \ldots, \Omega_8, \Omega_4',\Omega_6'\}$ agrees with the Yoshikawa moves, and introduce new $\{\Omega_9, \Omega_9',\Omega_{10},\Omega_{11},\Omega_{11}',\Omega_{12},\Omega_{12}'\}$ moves and prove that two surface-links immersed in $\mathbb{R}^4$ are equivalent if and only if, their singular marked graph diagrams are related by a combination of moves from the set $\mathcal D$.
\begin{figure}
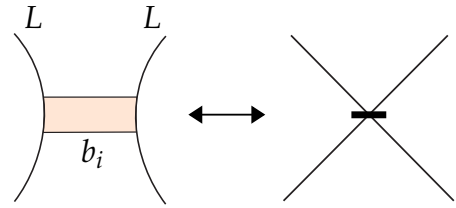

	\begin{center}
		\begin{lpic}[t(0.5cm),b(0.2cm)]{michal_07(5.8cm)}
			\lbl[t]{12,10;$b_i$}
			\lbl[b]{3,27;$L$}
			\lbl[b]{21,27;$L$}
		\end{lpic}
	\end{center}
	\caption{The band--vertex correspondence.}\label{michal_07}
	
\end{figure}

\begin{proof}[Proof of theorem \ref{tw:mainB}]
	
	By an ambient isotopy of $\mathbb R^3$, we shorten the bands of a singular link with bands $LB$ so that each band is contained in a small $2$-disk. Replacing the neighborhood of each band with the neighborhood of a $4$-valent marked vertex as in Fig.\;\ref{michal_07}, we obtain a singular marked graph. The set of Yoshikawa moves $\{\Omega_1, \ldots, \Omega_8, \Omega_4',\Omega_6'\}$ already generates all banded unlink moves when the singular marked graph is projected to the plane (see \cite{KeaKur08}). Then we consider a singular vertex also as a $4$-valent vertex, and use planar moves for a spatial rigid-vertex isotopy of \cite{Kau89} to translate the rigid-vertex isotopy of singular banded unlink moves involving a singular vertex to the corresponding planar moves.
	
	\begin{figure}[ht]
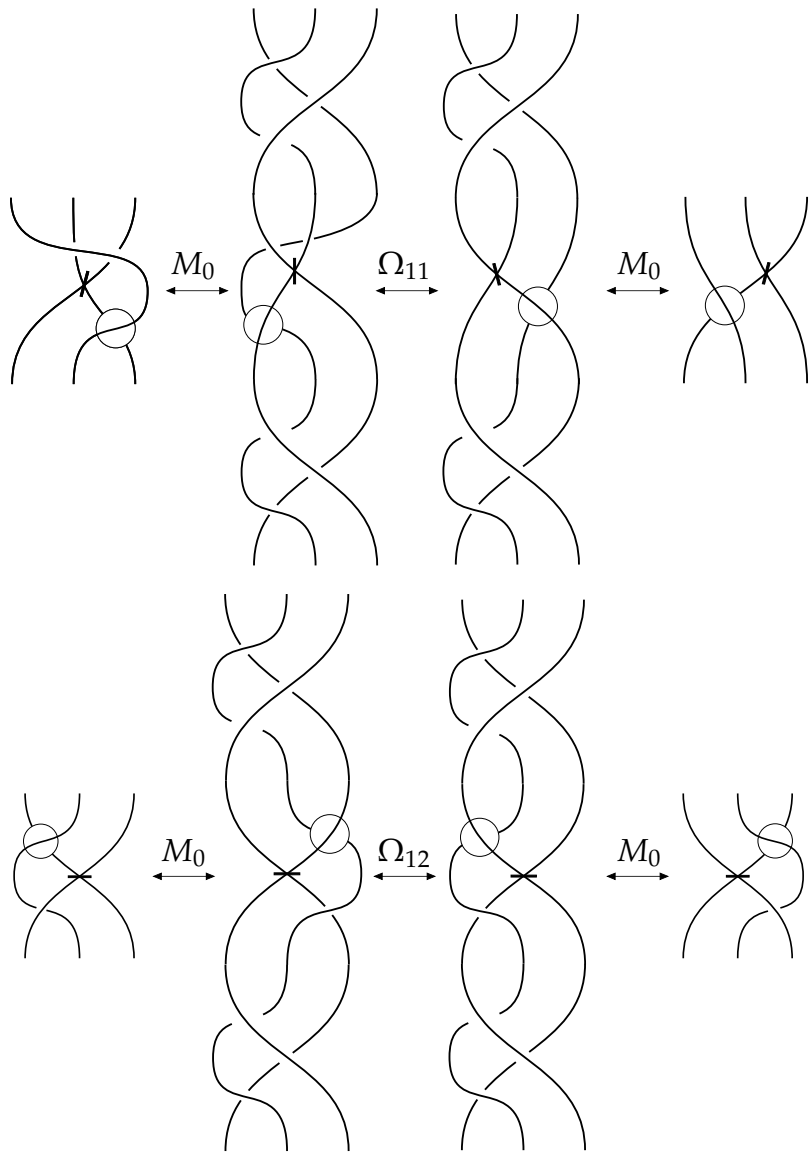

		\begin{center}
			\begin{lpic}[]{michal_05c(10.5cm)}
				\lbl[b]{39,187;$M_0$}
				\lbl[b]{84,187;$\Omega_{11}$}
				\lbl[b]{134,187;$M_0$}
				\lbl[b]{37,61;$M_0$}
				\lbl[b]{84,61;$\Omega_{12}$}
				\lbl[b]{134,61;$M_0$}
			\end{lpic}
			\caption{Deriving moves $\Omega_{11}'$ and $\Omega_{12}'$.\label{michal_05c}}
		\end{center}
	\end{figure}
	
	This correspondence gives us the moves $\Omega_9, \Omega_9',\Omega_9'',\Omega_9''',\Omega_{10},\Omega_{10}'$ shown in Figures \ref{michal_05} and \ref{michal_05b}. The remaining non-spatial moves of singular banded unlink $M_5$ and $M_6$ translate directly to the moves $\Omega_{11},\Omega_{11}'$ and $\Omega_{12},\Omega_{12}'$ respectively. Moreover, moves in Figure \ref{michal_05b} can be generated by the moves in Figure \ref{michal_05} as follows. The move $\Omega_{10}'$ can be obtained as a combination of moves from $\{\Omega_1, \Omega_9, \Omega_9', \Omega_{10}\}$ as in \cite[p. 62, Fig. 3.14]{Kam17} replacing the role there of a marker with a singular point. The same argument and figure show that the moves $\Omega_9''$ and $\Omega_9'''$ can be generated by the sets $\{\Omega_2, \Omega_9'\}$ and $\{\Omega_2, \Omega_9\}$ respectively. For the moves $\Omega_{11}'$ and $\Omega_{12}'$ involving both marked and singular vertices we need new arguments, the corresponding proofs are presented in Figure \ref{michal_05c}, where $M_0$ is a move generated by $\{\Omega_1, \Omega_2, \Omega_3, \Omega_4, \Omega_4',\Omega_9, \Omega_9',\Omega_{10}\}$ (i.e. corresponding to a spatial isotopy).
	
\end{proof}

\section{Group of an immersed surface-link}\label{sec4}

The \emph{group} of an immersed surface-link $F$ is $\pi_1(E(F))$ (i.e. the fundamental group of its complement in $\mathbb{R}^4$), the group is finitely presented. It is well-known (\cite{GomSti99}, \cite{Akb16}) that we can read off, the presentation for $\pi_1(X)$ from a Kirby diagram (with $0$-framed $2$-handles) for the $4$-manifold $X$, namely the generators correspond to dotted circles, and relators are words in terms of generators read as you travel along framed components. When we intersect the disc bounded by the dotted circle we add the letter corresponding to that generator with the exponent equal to $\pm$ depending on from which side of the disc we approach the intersection. Therefore, from this fact and Proposition \ref{prop1} we can deduce the following.

For a singular marked graph diagram, we define its \emph{abstract orientation} of strands of the diagram (regardless of the orientability of the immersed surface-link) as a fixed orientation of its negative resolution.

\begin{proposition}
	For any singular marked graph diagram $D$ of an immersed surface-link $F$ with an abstract orientation of its strands, the group of $F$ is generated by the connected components of the negative resolution of $D$ and the relations presented in Figure \ref{michal_15}.
\end{proposition}

\begin{figure}[h]
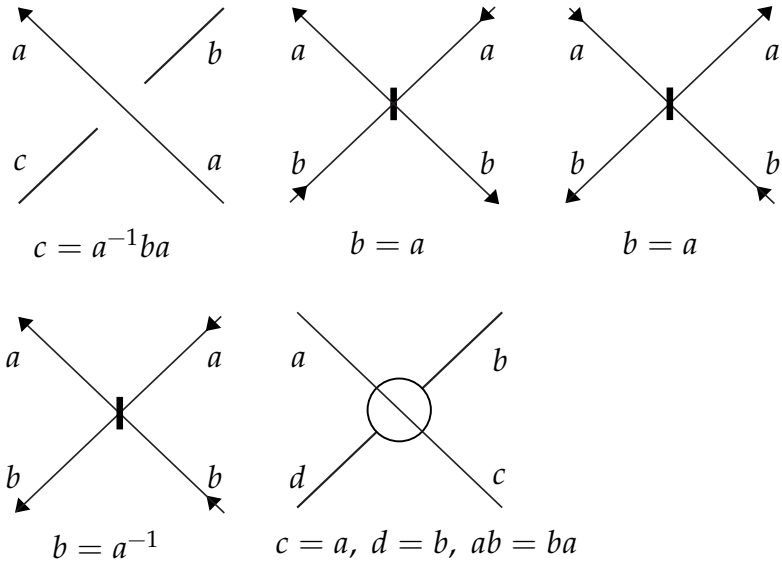

	\begin{center}
		\begin{lpic}[b(0.8cm),l(0.2cm),r(0.2cm)]{michal_15(10cm)}
			\lbl[r]{2,75;$a$}
			\lbl[l]{31,75;$b$}
			\lbl[r]{47,75;$a$}
			\lbl[l]{75,75;$a$}
			\lbl[r]{92,75;$a$}
			\lbl[l]{121,75;$a$}
			\lbl[r]{2,57;$c$}
			\lbl[l]{31,57;$a$}
			\lbl[r]{47,57;$b$}
			\lbl[l]{75,57;$b$}
			\lbl[r]{92,57;$b$}
			\lbl[l]{121,57;$b$}
			\lbl[l]{3,44;$c=a^{-1}ba$}
			\lbl[l]{54,44;$b=a$}
			\lbl[l]{98,44;$b=a$}
			\lbl[r]{1,25;$a$}
			\lbl[l]{31,25;$a$}
			\lbl[r]{47,25;$a$}
			\lbl[l]{77,25;$b$}
			\lbl[r]{1,6;$b$}
			\lbl[l]{31,6;$b$}
			\lbl[r]{47,6;$d$}
			\lbl[l]{77,6;$c$}
			\lbl[l]{6,-5;$b=a^{-1}$}
			\lbl[l]{42,-5;$c=a,\; d=b,\; ab=ba$}
		\end{lpic}
	\end{center}
	\caption{Relations near crossings and vertices.}\label{michal_15}
	
\end{figure}

\section{Oriented moves for oriented surface-links}\label{sec5}

We will introduce a generating set of moves on oriented singular marked diagrams.

\begin{theorem}
	Two oriented immersed surface-links $\mathcal{L}_1$ and $\mathcal{L}_2$  are equivalent if and only if their singular marked diagrams are related by a finite sequence of moves $\Gamma_1, \ldots, \Gamma_{12}d$ shown in Figures \ref{michal_oriA}--\ref{michal_oriE}.
\end{theorem}

\begin{figure}[ht]
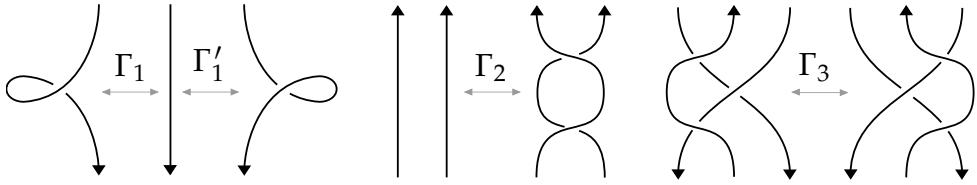

	\begin{center}
		\begin{lpic}[]{michal_oriA(12.75cm)}
			\lbl[b]{24,19;$\Gamma_1$}
			\lbl[b]{39,19;$\Gamma_1'$}
			\lbl[b]{93,19;$\Gamma_2$}
			\lbl[b]{155,19;$\Gamma_3$}
			
		\end{lpic}
		\caption{Moves, part I.\label{michal_oriA}}
	\end{center}
\end{figure}

\begin{figure}[ht]
	\begin{center}
		\begin{lpic}[]{michal_oriB(12.75cm)}
			\lbl[b]{35,60;$\Gamma_4$}
			\lbl[b]{97,60;$\Gamma_4'$}
			\lbl[b]{155,60;$\Gamma_5$}
			\lbl[b]{22,20;$\Gamma_6$}
			\lbl[b]{35,20;$\Gamma_6'$}
			\lbl[b]{83,19;$\Gamma_7$}
			\lbl[b]{149,19;$\Gamma_8$}
		\end{lpic}
		\caption{Moves, part II.\label{michal_oriB}}
	\end{center}
\end{figure}

\begin{figure}[ht]
	\begin{center}
		\begin{lpic}[]{michal_oriC(12.75cm)}
			\lbl[b]{28,20;$\Gamma_9$}
			\lbl[b]{93,20;$\Gamma_9'$}
			\lbl[b]{156,20;$\Gamma_{10}$}
		\end{lpic}
		\caption{Moves, part III.\label{michal_oriC}}
	\end{center}
\end{figure}

\begin{figure}[ht]
	\begin{center}
		\begin{lpic}[]{michal_oriD(11.5cm)}
			\lbl[b]{35,62;$\Gamma_{11}a$}
			\lbl[b]{112,62;$\Gamma_{11}b$}
			\lbl[b]{35,20;$\Gamma_{11}c$}
			\lbl[b]{112,20;$\Gamma_{11}d$}
		\end{lpic}
		\caption{Moves, part IV.\label{michal_oriD}}
	\end{center}
\end{figure}

\begin{figure}[ht]
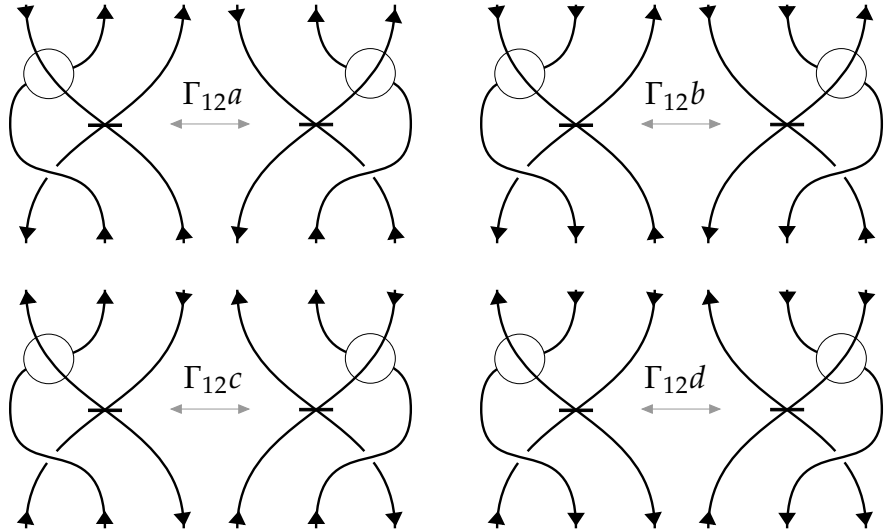

	\begin{center}
		\begin{lpic}[]{michal_oriE(11.5cm)}
			\lbl[b]{30,62;$\Gamma_{12}a$}
			\lbl[b]{98,62;$\Gamma_{12}b$}
			\lbl[b]{30,20;$\Gamma_{12}c$}
			\lbl[b]{98,20;$\Gamma_{12}d$}
		\end{lpic}
		\caption{Moves, part V.\label{michal_oriE}}
	\end{center}
\end{figure}

\begin{proof}
	The generating set of oriented Reidemeister moves in Figure \ref{michal_oriA} was introduced in \cite{Pol10}. The generating set of oriented moves on marked graph diagrams in Figure \ref{michal_oriB} was introduced in \cite{JKL15}. Here, for the move $\Gamma_6$ we switch the orientation of the strand (they can be replaced as shown also in \cite{JKL15}) to make pictures more simple and consistent with move $\Gamma_1$.
	\par
	Moves $\Gamma_9$, $\Gamma_9'$, and $\Gamma_{10}$ in Figure \ref{michal_oriC} viewed as a (classical) singular diagrams, generate all other possible orientations on this (unoriented) moves by \cite{BEHY18}. We checked, that in the process of deriving those moves the singular rigid crossing, from starting to ending picture, the singular position was not rotated by $90^{\circ}$, otherwise it will produce the other type of our singular circle.  
	\par
	To complete the proof we show now the moves $\Gamma_{11}a, \ldots, \Gamma_{12}d$ in Figures \ref{michal_oriD}--\ref{michal_oriE} that generate all possibilities of orienting the unoriented moves (i.e. the same moves but without orientation) derived earlier in this paper.
	
\end{proof}

\section{Fundamental quandle}\label{sec6}

A \emph{quandle} is a non-empty set $Q$ together with an operation $:Q\times Q\to Q$ that satisfies the following conditions:
\begin{enumerate}
	\item $\forall_{a\in Q}\;$ $a*a=a\quad\quad$ (\emph{idempotency})
	\item $\forall_{a, b\in Q}\;$ $\exists!{c\in Q};$ $c*a=b\quad\quad$ (\emph{right invertibility})
	\item $\forall_{a, b, c\in Q}\;$ $(a*b)*c=(a*c)*(b*c)\quad\quad$ (\emph{self-distributivity})
\end{enumerate}

Condition (2) can be equivalently replaced by the condition:
\begin{enumerate}[(2')]
	\item $\forall_{a\in Q};$ the transformation $* a:Q\to Q$ defined as $x\mapsto x* a$ is a bijection.
\end{enumerate}

A generalization of quandles (called biquandles) was introduced in \cite{KauRad03}. A biquandle is an algebraic structure with two binary operations satisfying certain conditions which can be presented by semi-arcs of links (or semi-sheets of surface-links) as its generators modulo oriented Reidemeister moves (or Roseman moves). In \cite{KKKL18b} S. Kamada, A. Kawauchi, J. Kim, and S. Y. Lee discussed the (co)homology theory of biquandles and developed the biquandle cocycle invariants for oriented surface-links by using broken surface diagrams generalizing quandle cocycle invariants. Then showed how to compute the biquandle cocycle invariants from marked graph diagrams.
\par

Let us fix a quandle $Q$, let $D$ be an oriented singular marked diagram of an oriented immersed surface-link, and let $\Lambda$ be the set of components of this diagram. A \emph{quandle coloring} ${\mathcal C}$ is a mapping ${\mathcal C}:\Lambda\to Q$ such that around each classical, marked or singular point, the relation shown in Figure \ref{michal_biq} holds. These conditions are consistent around each classical, marked or singular point due to the axioms for the quandle. Denote by ${Col}_Q(D)$ the set of all colorings of the diagram $D$ with quandle $Q$.

\begin{figure}[ht]
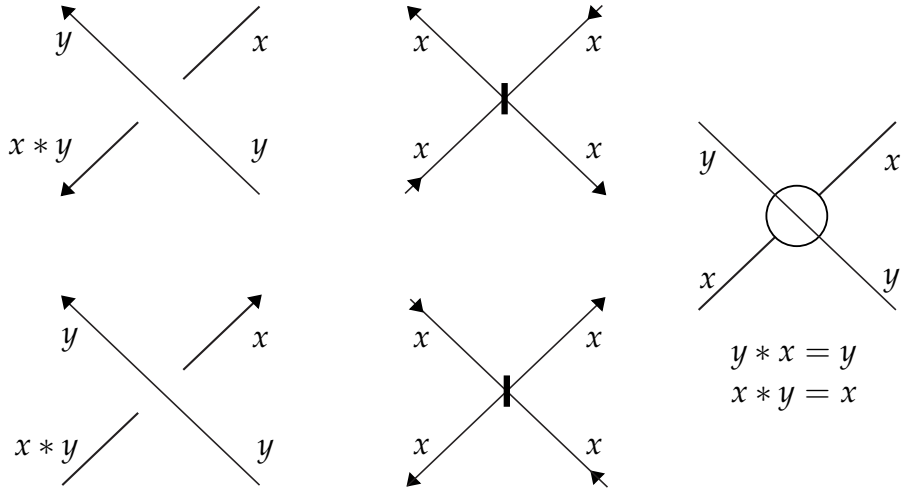

	\begin{center}
		\begin{lpic}[l(0.6cm),b(0.6cm),r(0.4cm)]{michal_biq(11cm)}
			\lbl[r]{2,75;$y$}
			\lbl[l]{32,75;$x$}
			\lbl[r]{62,75;$x$}
			\lbl[l]{88,75;$x$}
			\lbl[r]{62,25;$x$}
			\lbl[l]{88,25;$x$}
			\lbl[r]{2,57;$x*y$}
			\lbl[l]{32,57;$y$}
			\lbl[r]{62,57;$x$}
			\lbl[l]{88,57;$x$}
			\lbl[r]{62,7;$x$}
			\lbl[l]{88,7;$x$}
			\lbl[r]{3,25;$y$}
			\lbl[l]{32,25;$x$}
			\lbl[r]{3,6;$x*y$}
			\lbl[l]{33,6;$y$}
			\lbl[r]{110,55;$y$}
			\lbl[l]{138,55;$x$}
			\lbl[r]{110,35;$x$}
			\lbl[l]{138,35;$y$}
			\lbl[t]{123,24;$y*x = y$}
			\lbl[t]{123,17;$x*y = x$}
		\end{lpic}
		\caption{Labeling conditions.\label{michal_biq}}
	\end{center}
\end{figure}

\begin{proposition}
	The quandle axioms are chosen such that given a quandle coloring of one side of any $\Gamma$ type move, there is a unique quandle coloring of the other side of the move with the condition that colors agree on the boundary arcs that leaves the disc where the move is performed. 
\end{proposition}

\begin{proof}
	The proof for the moves $\Gamma_1, \ldots, \Gamma_3$ can be found in the existing literature, the axioms for the quandle were motivated to satisfy those Reidemeister moves, for moves $\Gamma_4, \ldots, \Gamma_8$ see for example \cite{JouNel20}. We checked (easily by hand) the proof for the remaining moves.
\end{proof}

\begin{corollary}
	The number of quandle colorings $\#Col_Q(D)$ of an oriented singular marked diagram $D$ is an invariant of an oriented surface-link $\mathcal{L}$ presented by $D$, we can denote it therefore by $\#Col_Q(\mathcal{L})$.
\end{corollary}

\begin{example}
	
	Let us consider the Fenn--Rolfsen link $FR$, shown as the singular marked diagram in Figure \ref{fr}.  Let a quandle on $Q=\{1, 2, 3, 4\}$ be the quandle with the operation given by the following matrix.
	\renewcommand*{\arraystretch}{1.4}	
	\[
	\begin{array}{r|rrrr} 
		* & 1 & 2 & 3 & 4\\\hline
		1 & 1 & 1 & 2 & 2\\
		2 & 2 & 2 & 1 & 1\\
		3 & 4 & 4 & 3 & 3\\
		4 &	3 & 3 & 4 & 4
	\end{array}
	\]
	
	For a quandle coloring, we need an orientation. But the given quandle $Q$ is an involutory quandle so we don't need an orientation. We can enumerate that $\#Col_Q(FR)= 16$.
\end{example}

\begin{figure}[ht]
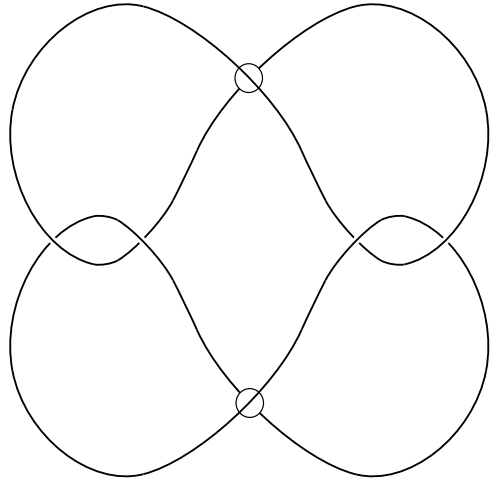

	\begin{center}
		\begin{lpic}[]{fr(7cm)}
			
		\end{lpic}
		\caption{The Fenn--Rolfsen link.\label{fr}}
	\end{center}
\end{figure}

\end{document}